\documentclass{amsart}

\usepackage{amssymb,amscd, amsfonts,  amsthm}

\usepackage{a4wide}

\newtheorem{theorem}{Theorem}
\newtheorem{lemma}[theorem]{Lemma}
\newtheorem{corollary}[theorem]{Corollary}
\newtheorem{proposition}[theorem]{Proposition}

\theoremstyle{definition}
\newtheorem{definition}[theorem]{Definition}

\newcommand{\R}{\mathbb{R}}
\newcommand{\Rn}{\R^n}
\newcommand{\Q}{\mathbb{Q}}
\newcommand{\Qn}{\Q^n}
\newcommand{\GLnZ}{\mathsf{GL}(n,\Zed)\ltimes \Zed^n}
\newcommand{\GLeZ}{\mathsf{GL}(e,\Zed)\ltimes \Zed^e}
\newcommand{\Zed}{\ensuremath{\mathbb{Z}}}

\newcommand{\den}{\mathrm{den}}
\newcommand{\conv}{\mathrm{conv}}
\newcommand{\aff}{\mathrm{aff}}
\newcommand{\rank}{\mathrm{rank}}

\DeclareMathOperator{\orb}{\rm orb}

 \title[Classifying $(\GLnZ)$-orbits]
{Classifying   orbits 
of the affine group over the integers}

\author{Leonardo Manuel Cabrer and Daniele Mundici }
\thanks{This research was supported by a 
 Marie Curie Intra European Fellowship 
 within the 7th European Community Framework 
 Program (ref. 299401-FP7-PEOPLE-2011-IEF)} 

\address[L.M. Cabrer]{Department of Statistics,
Computer Science and Applications,  ``Giu\-sep\-pe Parenti''\\ 
University of Florence\\
Viale Morgagni 59 --
50134\\ Florence \\
Italy}
\email{ l.cabrer@disia.unifi.it }

\address[D. Mundici]{Department of
Mathematics and Computer Science  ``Ulisse Dini'' \\
University of Florence\\
Viale Morgagni 67/A \\
I-50134 Florence \\
Italy}
\email{ mundici@math.unifi.it }

\begin{document}


\begin{abstract}
For each  $n=1,2,\dots$, let  $\GLnZ$ be the affine group
over the integers.
For every  point $x=(x_1,\dots,x_n) \in \Rn$ let  
$\orb(x)=\{\gamma(x)\in \Rn\mid
\gamma\in \GLnZ\}.$ 
Let $G_{x}$ be the subgroup of the additive group
$\mathbb R$ generated by $x_1,\dots,x_n, 1$. 
If    $\rank(G_x)\neq n$ 
then 
$\orb(x)=\{y\in\Rn\mid G_y=G_x\}$. Thus,
$G_x$ is a complete classifier of $\orb(x)$.
By contrast, if  $\rank(G_x)=n$, knowledge of 
$G_x$ alone is not   sufficient in general 
to uniquely recover  $\orb(x)$: as a matter of fact,  $G_x$
determines precisely  $\max(1,\frac{\phi(d)}{2})$
different orbits, where $d$ is the denominator
of the smallest positive nonzero rational in $G_x,$
and $\phi$ is Euler function.
To get a complete classification, rational polyhedral geometry 
provides  an integer  $1\leq c_x\leq \max(1,d/2)$  such that
$\orb(y)=\orb(x) $    iff
$(G_{x},c_{x})=(G_{y},c_{{y}})$.   
\end{abstract}

\maketitle

%%%%%%%%%%%%%%%%%%%%%%%%%%%%%
\section{Introduction} 
%%%%%%%%%%%%%%%%%%%%%%%%%%%%%
Throughout we   let
$\GLnZ$ 
denote the group of affine  transformations of the form
$
x\mapsto \mathcal{U}x + t \,\, \mbox{for} \,\, x\in \mathbb R^n,
$
 where $t\in \mathbb Z^n$
 and  $\mathcal U$ is an integer  $n \times n$ matrix with
 $\det(\mathcal U)=\pm 1$.   
The dimension  of the ambient space  $\Rn$ will always be clear from
the context.
We let $\orb(x)$  denote the
{\it $(\GLnZ)$-orbit} of $x\in \Rn,$
\[
\orb(x)=\{\gamma(x)\in \Rn\mid
\gamma\in \GLnZ\}.
\]

In this self-contained  paper, generalizing  \cite{Mun201X}
we  classify
$(\GLnZ)$-orbits for all $n$:   
every point in  $\Rn$ is assigned an invariant,
in such a way that two points have the same $(\GLnZ)$-orbit iff
they have the same invariant.
For each $x=(x_1,\ldots,x_n)\in \Rn$,
the  main invariant considered in this paper is the group
\[
G_x=\mathbb Z+\mathbb Zx_1+\mathbb Zx_2+\dots+\mathbb Zx_n
\]
generated by $1,x_1,x_2,\ldots,x_n$ in the
 additive group
$\mathbb R$.
From
Corollaries \ref{corollary:semifinal} and
\ref{corollary:final} it follows that
  $G_x$  
 completely classifies  the $(\GLnZ)$-orbit of any
 $x\in\Rn$ when 
$\rank(G_x)\neq n$.

If  $\rank(G_x)=n$,
\,  $G_x$   determines  precisely  $\max(1,\frac{\phi(d)}{2})$
different $(\GLnZ)$-orbits, where $d$ is the denominator
of the smallest positive nonzero rational in $G_x$
and $\phi$ is Euler function. To get a complete
  invariant for $\orb(x)$, rational polyhedral geometry equips $x$ with 
   an integer 
    $1\leq c_x\leq \max(1,d/2)$ 
     which, together with
   $G_x$, completely classifies $\orb(x)$. Thus
   in Theorem \ref{theorem:classification} we prove: 
     $\orb(x)=\orb(y)$  iff
$(G_x,c_x)=(G_y,c_y).$
   In case  $\rank(G_x)\not=n,\,\,\, c_x=1$.

While the study of the orbits  of 
the groups   $\mathsf{GL}(2,\mathbb Z)$,   
$\mathsf{SL}(2,\mathbb Z),$
$\mathsf{SL}(2,\mathbb Z)_+$ 
requires techniques from various mathematical areas 
\cite{dan, gui, launog, nog2002, nog2010, wit},
a main possible reason of interest in our 
classification  stems
from  the pervasive and novel  role played by
 rational polyhedral geometry  \cite{Ew1996},
 through the fundamental notion of a
(Farey) regular simplex \cite{mun-cpc}.
Regular simplexes in $\Rn$ are the affine counterparts
of regular cones in $\mathbb Z^{n+1}.$  Regular fans
 are complexes
of regular cones, and yield a combinatorial classification 
of  nonsingular toric varieties~\cite{Ew1996}.

%%%%%%%%%%%%%%%%%%%%%%%%%%%%%%
\section{Preliminaries: affine spaces and regular simplexes}
\label{section:preliminaries}
%%%%%%%%%%%%%%%%%%%%%%%%%%%%%%
A point $y=(y_1,\ldots,y_n) \in\Rn$
is said to be   {\it rational}   if each coordinate $y_i$ is
a rational number.

A {\it rational hyperplane} $H \subseteq \mathbb R^n$
is  a set of the form 
$ 
H =\{z \in \mathbb R^{n}   \mid  \langle h, z\rangle = r\},
$ 
    for some nonzero vector
     $h\in \Qn$ and $r\in \mathbb Q$.
     Here $\langle \mbox{-},\mbox{-}\rangle$ denotes scalar product.
A {\it rational affine space} $A$ in  $\mathbb R^n$
is  an  intersection
of rational hyperplanes in $\mathbb R^n$.
For any  $S\subseteq \mathbb R^n$
  the   {\it affine hull}  
 $\,\,\aff(S)\,\,$    is the set of all {\it affine combinations} in  $\mathbb R^n$
of elements of $S$.  Thus  $z$ belongs to   $\aff(S)$
  iff there are   $w_1,\ldots,w_k\in S$
and   $\lambda_1,\ldots,\lambda_k \in \mathbb R$
such that $\lambda_1+\dots+\lambda_k=1$
and $z=\lambda_1\cdot w_1+\dots+\lambda_k\cdot w_k$.
A set  $\{y_1,\ldots,y_m\}$ of  points in
$\mathbb R^n$ is said to be
{\it affinely independent}  if    none of its  elements 
 is  an affine combination of the remaining elements.
 One then easily sees that a subset $E$ of  $\Rn$
is  a rational affine subspace
 of $\Rn$ iff there exist $v_0,\ldots,v_m\in\Qn$ such that 
 $E=\aff(v_0,\ldots,v_m)$.
 For $\,\,0\leq m\leq n$, an {\it m-simplex}
in $\mathbb R^{n}$ is the
{\it convex hull} $T = \conv(v_{0},\ldots,v_{m})$ of $m+1$ affinely
independent points $v_{0},\ldots,v_{m}\in \mathbb R^n$.  The {\it  vertices}
$v_{0},\ldots,v_{m}$ are uniquely determined by $T$.
$T$ is said to be a {\it rational simplex}  if its vertices are
rational.

By  the  {\it denominator}  $\den(y)$ of a rational point
  $y$
 we understand the least common denominator 
  of its coordinates. 
The  {\it homogeneous correspondent}
of a rational point $y=(y_1,\ldots,y_n)\in \Rn$ is the %primitive
integer  vector 
 \[
 \widetilde{ y } = (\den(y)\cdot y_1,
 \ldots, \,\den(y)\cdot y_n,\,\,\den(y)) 
 \in \mathbb Z^{n+1}.
  \]

\noindent
Following \cite{mun-cpc},
a rational $t$-simplex  
$T=\conv(v_0,\ldots,v_t)\subseteq \Rn$
is said to be  {\it (Farey) regular} 
if the set $\{\tilde v_0 ,\ldots, \tilde v_t\}$ of 
homogeneous correspondents of the vertices of $T$
can be extended to a   base  of
the free abelian group $\mathbb Z^{n+1}$.
(Regular simplexes are called ``unimodular'' in \cite{mun-dcds}.)

 \begin{lemma}
       \label{Lemma:GammaFromSimplex}
For each rational point 
$x\in\Rn$
and  $\psi\in\GLnZ$, 
$\den(x)=\den(\psi(x)).$
      Moreover, if
$\conv(v_{0},\ldots,v_{n})\subseteq \mathbb R^{n}$ and
$\conv(w_{0},\ldots,w_{n})\subseteq \mathbb R^{n}$ 
are regular
$n$-simplexes, the following conditions are equivalent:
\begin{itemize}
\item[(i)] $\den(v_{i})=\den(w_{i})$ for all
$i\in\{0,\ldots,n\}$;

\smallskip
\item[(ii)] there exists $\gamma\in\GLnZ$ such that $\gamma(v_i)=w_i$ for each $i\in\{0,\ldots, n\}$.
\end{itemize}
\end{lemma}

\begin{proof}
It is easy to see that $\den(x)=\den(\psi(x)),$ for every
$\psi\in \GLnZ.$
So
(ii)$\Rightarrow$(i) is immediately verified.
For the converse direction (i)$\Rightarrow$(ii), the  regularity
of 
$\conv(v_{0},\ldots,v_{n})$ and
$\conv(w_{0},\ldots,w_{n})$ 
yields
    bases $\{\tilde{v}_{0},\ldots,
\tilde{v}_{n}\}$ and $\{\tilde{w}_0,\ldots,\tilde{w}_n\}$
  of the free
abelian group $\mathbb Z^{n+1}$. 
There are   integer matrices
$\mathcal M$ and $\mathcal N$
such that $\mathcal M \tilde{v}_i=\tilde{w}_i$
 and 
$\mathcal N \tilde{w_i}=\tilde{v}_i$ for each $i\in\{0,\ldots,n\}$.
It follows that $\mathcal M\cdot \mathcal N$ and 
$\mathcal N\cdot \mathcal M$ are equal to the identity matrix. 
For each 
$i\in\{0,\ldots,1\}$ the
 last coordinate of $\tilde{v}_i$, as well as the last coordinate
 of $\tilde{w}_i$,  coincide
 with  $\den(v_i)=\den(w_i)$. Thus $\mathcal M$ and $\mathcal N$ 
  have the form
\[
\mathcal{M}=\left(\begin{tabular}{c|c} $\mathcal U$ & $b$ \\ \hline $0,\ldots,0 $ &
$1$ \end{tabular}\right) \quad \mathcal{N}=\left(\begin{tabular}{c|c} $\mathcal V$ & $c$ \\ \hline $0,\ldots,0 $ &
$1$ \end{tabular}\right) 
\] 

\medskip
\noindent for some  integer $(n\times n)$-matrices $\mathcal U$ and $\mathcal V$
and integer vectors $b,c\in\mathbb Z^n$.  
For each  $i=0,\ldots,n$ the following holds:
$\mathcal{M}\tilde{v}_{i}= \mathcal M(\den(v_{i})\cdot(v_{i},1))
=\den(w_{i})\cdot(\mathcal Uv_{i}+b,1).  $ 
From
$\mathcal M\tilde{v}_{i}=\tilde{w}_{i}=\den(w_{i})\cdot(w_{i},1)$ 
we obtain
$\mathcal Uv_{i}+b=w_{i}$.
Similarly,  $\mathcal Vw_i +c=v_i$. 
In conclusion,
both maps $\gamma(x)=\mathcal Ux+b$ 
and $\mu(y)=\mathcal Vy+c$ are in $\GLnZ$.
Further, 
$\gamma^{-1}=\mu$ and $\gamma(v_i)=w_i$ for each $i\in\{0,\ldots, n\}$.
\end{proof}

\begin{lemma}\label{lemma:RegularInAffineSpace}
Let $F\subseteq \Rn$ be
an $e$-dimensional rational affine space
$(e=0,\dots,n)$,
and  $v_0$  a rational point lying in $F$. Then there
 exist rational points $v_1,\ldots,v_e\in F$ such that the simplex
 $\conv(v_0,\ldots,v_e)$ is regular.
\end{lemma}
\begin{proof}
The primitive integer vector $\tilde v_0\in \Rn$ is part of a
base  $(b_0=\tilde{v}_0,b_1,\dots,b_n)$  of  the free abelian group
$\mathbb Z^{n+1}$.
We may choose the integer vectors
   $b_1,\dots,b_e$ in such a way that, 
for each  $i=1,\dots,e$, 
the uniquely determined rational
point  $v_i$  in $\Rn$ with   $\tilde v_i=b_i$
belongs to $F$. It follows that 
$\conv(v_0,\ldots,v_e)$ is a regular
$e$-simplex contained in $F$.  
\end{proof}

%%%%%%%%%%%%%%%%%%%%%%%%%
\section{Classification of rational affine subspaces of $\Rn$}
%%%%%%%%%%%%%%%%%%%%%%%%%%

Given two rational affine spaces $F,G\subseteq \Rn$,
the main result of this section 
(Theorem~\ref{Theorem:ClassAffineSpace})  provides
a necessary and sufficient condition for the existence of 
$\gamma\in\GLnZ$ such that $\gamma(F)=G$. 
This is a first step for our complete classification 
of $(\GLnZ)$-orbits  in Section~\ref{Section:Class}.

\begin{definition}\label{Def:DF}
For every  non-empty rational affine space $F\subseteq \Rn$ we let 
\[
\boxed{d_F=\min\{\den(v)\mid v\in  F\cap \Qn\}.}
\]
\end{definition}

We have the following useful characterization of $d_F:$
\begin{lemma}
\label{Lemma:CalculationofD}
Let $F\subseteq \Rn$ be an $e$-dimensional rational affine space
$(e=0,\dots,n)$. 
For any regular 
 $e$-simplex $\conv(v_0,\ldots,v_e)\subseteq F$, \,\,\,
  $d_F=\gcd(\den(v_0),\ldots, \den(v_e))$.
\end{lemma}

\begin{proof}
Let $d= \gcd(\den(v_0),\ldots, \den(v_e))$. 
For each  $v\in F\cap\Qn$, $\tilde v\in \Zed\tilde v_0+\cdots+ \Zed\tilde v_e$. 
Since  $\den(v)$ is a  linear combination of  $\den(v_0),\ldots, \den(v_e)$
with integer coefficients, then
 $d\leq d_F$. Conversely, for suitable integers
$m_0,\ldots, m_e$
we can write 
$d=m_0\cdot \den(v_0)+\cdots+m_e \cdot \den(v_e)$. 
The uniquely determined rational point  $w\in F\cap \Qn$ 
with  $\tilde w=m_0\cdot\tilde v_0+\cdots+m_e\cdot \tilde v_e$ 
  satisfies    $\den(w)=d$. Thus $d_F\leq d$.
\end{proof}

\begin{lemma}\label{Lemma:HomogeneousSimplex}
Let $F\subseteq \Rn$ be an $e$-dimensional rational affine space
$(e=0,\dots,n)$. 
Then
 there are rational points $v_0,\ldots,v_e$ lying in $F$ 
  such that  
 $\conv(v_0,\ldots,v_e)$ is a  regular  $e$-simplex and $\den(v_i)=d_F$
  for each  $i\in\{0,\dots,e\}$.
\end{lemma}

\begin{proof}
Pick  $v_0\in F\cap \Qn$  with  $\den(v_0)=d_F$.
 Lemma~\ref{lemma:RegularInAffineSpace} yields rational points
$w_1,\ldots,w_e\in F$ such that $\conv(v_0,w_1,\ldots,w_e)$ is regular. 
If 
$\den(w_i)>d_F$
for some $i\in\{1,\ldots,e\}$,  then
let   the integer $m_i$  be uniquely determined by writing
  $m_i \cdot d_F<\den(w_i)\leq (m_i+1) \cdot  d_F$.
 Let $v_i\in F\cap\Qn$ be the unique rational point 
 with  $\tilde v_i=\tilde w_i-m_i \tilde v_0$. Then 
$d_F\leq\den(v_i)=\den( w_i)- m_i \cdot 
 \den(v_0)=\den( w_i)- m_i  \cdot  d_F\leq d_F.$
The  new $(e+1)$-tuple of integer vectors 
$(\tilde v_0,\tilde v_1,\ldots,\tilde v_e)$ can be completed to a base of $\Zed^{n+1}$,
precisely  as 
 $(\tilde v_0,\tilde w_1\ldots,\tilde w_e)$ does.
Thus  $\conv(v_0,\ldots,v_e)$ is regular.
\end{proof}

\begin{definition}
\label{definition:c}
Let $F\subseteq \Rn$ be an
 $e$-dimensional
 rational affine space. If   $0\leq e <n$
 we define 
\[
\boxed{
c_F=\min\{\den(v)\mid v\in \mathbb Q^n\setminus
 F\mbox{ and $\exists$ } v_0,\ldots, v_e\in F\cap\Qn\mbox{ with }\conv(v,v_0,\ldots,v_e)\mbox{ regular} \}.
 }
 \]
If $e=n$ we define  $c_F=1$.
\end{definition}

\medskip
\begin{lemma}
\label{Lemma:FullSimplexControlledDenominators}

Let $F\subseteq \Rn$ be an $e$-dimensional rational affine space
$(e=0,\dots,n)$. Then  there are rational points $v_0,\ldots,v_n$ such that
\begin{itemize}
\item[(i)]  $v_0,\ldots,v_e\in F$; 
\item[(ii)] $\den(v_i)=d_F$\,\, for all $i\in\{0,\dots,e\}$;
\item[(iii)] $\den(v_i)=c_F$\,\, for all $i\in\{0,\dots,n\}\setminus\{0,\dots,e\}$;
\item[(iv)] $\conv(v_0,\ldots,v_n)$ is a regular $n$-simplex (whence in particular, 
$v_{e+1},\dots,v_n\notin F$).
\end{itemize}
\end{lemma}

\begin{proof}
In case  $e=n$, 
the result follows directly from Lemma~\ref{Lemma:HomogeneousSimplex}. 
In case $e<n$,  
let $w_0,\ldots,w_n\in \Qn$ be such that $w_0,\ldots,w_e\in F$, $\conv(w_0,\ldots,w_n)$ is regular,  and $\den(w_{e+1})=c_F$.
  Lemma~\ref{Lemma:HomogeneousSimplex}
  yields rational points  $v_0,\ldots,v_e\in F$ such that
  $\conv(v_0,\ldots,v_e)$ is regular
  and \,\, $\den(v_i)=d_F$ for each $0\leq i\leq e$.
  For  every  $v\in F\cap\Qn$ we can write 
 $\tilde v\in \Zed\tilde v_0+\cdots+ \Zed\tilde v_e$, whence in particular $\Zed\tilde v_0+\cdots+ \Zed\tilde v_e=\Zed\tilde w_0+\cdots+ \Zed\tilde w_e$. Therefore, $\{\tilde v_0,\ldots,\tilde v_e,\tilde w_{e+1},\ldots,\tilde w_n\}$ 
  is a base of $\Zed^{n+1}$,  and the $n$-simplex
  $T_0=\conv(v_0,\ldots,v_e,w_{e+1},\ldots,w_n)$ is regular.
In order  to obtain from $w_{e+1},\ldots,w_n$ the points
  $v_{e+1},\ldots,v_n$ satisfying  (iii)-(iv), we preliminarily set 
 $v_{e+1}=w_{e+1}.$
Similarly,  for each  $i=e+2,\dots,n$ such that $\den(w_i)=c_F$
we set $v_i=w_i.$ 
If $c_F=\den(w_{e+2})=\dots=\den(w_{n})$,  we are done.
 Otherwise, for all $i\in\{e+2,\ldots, n\}$ with   \,\,$\den(w_i)>c_F$,
 letting the integer  $m_i$ be the unique solution of  
   $m_i\cdot c_F<\den(w_i)\leq (m_i+1)\cdot c_F$,  we
let   the rational point  $v_i\in\Rn$ be 
uniquely determined by the stipulation 
    $\tilde v_i=\tilde w_i-m_i \cdot\tilde w_{e+1}$.  We then have
$
c_F\leq\den(v_i)=\den(w_i)- m_i\cdot \den(w_{e+1})=\den(w_i)- m_i\cdot c_F\leq c_F.
$
The $n$-simplex  $T=\conv(v_0,\dots,v_n)$ obtained from 
$T_0$ in this way is regular, 
and has the desired properties.
\end{proof}

\medskip
The following theorem yields the classification of $(\GLnZ)$-orbits of
 rational affine spaces: the complete classifier is just a triplet of
 positive integers. As usual, for every affine space
 $F\subseteq \Rn$,
  $\dim(F)$ denotes its dimension.

\medskip
\begin{theorem} 
\label{Theorem:ClassAffineSpace}
Let $F,G\subseteq \Rn$ be non-empty rational affine spaces. Then
 following conditions are equivalent:
\begin{itemize}
\item[(i)] There exists $\gamma\in\GLnZ$ such that $\gamma(F)=G$;

\smallskip
\item[(ii)] $(\dim(F),d_F,c_F)=(\dim(G),d_G,c_G)$.
\end{itemize} 
\end{theorem}
\begin{proof}
(i)$\Rightarrow$(ii) 
Since $\gamma\in\GLnZ$ is a bijective affine transformation, $\dim(F)=\dim(\gamma(F))=\dim(G)$. The identities $d_F=d_G$ and $c_F=c_G$ follow directly from Definitions~\ref{Def:DF} and~\ref{definition:c}, because   $\gamma$ preserves denominators (Lemma~\ref{Lemma:GammaFromSimplex}).

(ii)$\Rightarrow$(i). Let $v_0,\ldots,v_n\in  \Qn$ satisfy conditions (i)-(iv) of Lemma~\ref{Lemma:FullSimplexControlledDenominators} relative to $F$. Similarly let $w_0,\ldots,w_n\in  \Qn$ satisfy the same conditions  relative to $G$. By hypothesis, $\den(v_i)=d_F=d_G=\den(w_i)$ for each  
$0\leq i\leq \dim(F)=\dim(G)$ and  $\den(v_i)=c_F=c_G=\den(w_i)$ for each 
 $i\in \{0,\ldots,n\}\setminus\{0,\ldots, \dim(F)\}$.
  Lemma~\ref{Lemma:GammaFromSimplex}
yields  $\gamma\in\GLnZ$  such that $\gamma(v_i)=w_i$ for each $i\in\{0,\ldots, n\}$. Since $\gamma$ is an affine transformation and $F=\aff(v_0,\ldots,v_e)$, then $G=\aff(w_0,\ldots, w_e)=\aff(\gamma(v_0),\ldots,\gamma(v_e))=\gamma(F)$.
\end{proof}

As shown by the next three results,  
in many cases the parameter $c_F$ 
has 
no  role in the classification of  $(\GLnZ)$-orbits
of rational affine spaces.

\begin{lemma}
\label{lemma:expedient}
Let $F\subseteq \Rn$ be a  rational affine space. If $0\leq \dim(F)<n-1$
 then $c_F=1$. If $\dim(F)=n-1$, then $\gcd(c_F,d_F)=1$ and $1\leq c_F\leq 
 \max(1,d_F/2)$.
\end{lemma}

\begin{proof}
Choose points  $v_0,\ldots,v_n\in \Rn$
satisfying  conditions (i)-(iv) in Lemma~\ref{Lemma:FullSimplexControlledDenominators}. Since $\tilde v_0,\ldots,\tilde v_n$ is a base of $\Zed^{n+1}$, some
 linear combination of $\den( v_0),\ldots,\den( v_n)$ with integer coefficients 
equals~$1$. Thus  $\gcd(d_F,c_F)=1$.  
A moment's reflection shows that 
 $c_F\leq d_F$. So, in case  $d_F=1$ then $c_F=1$, and the result follows. 
On the other hand, in case  $d_F> 1$, from
 $\gcd(d_F,c_F)=1$ and $c_F\leq d_F$, it follows that $c_F<d_F$.
 
Let $e=\dim(F)$.  
 
 If $e<n-1$, then $v_{e+1}\neq v_{n}$ and (by Lemma~\ref{Lemma:FullSimplexControlledDenominators}(iii))
   $\den(v_{e+1})=
 \den(v_{n})=c_F$. Let $v \in \Qn$ be such that $\tilde v =\tilde v_{0}-\tilde v_{n}$.
  Then 
$\{\tilde v_0,\ldots,\tilde v_{n-1},\tilde v\}$ is a base of $\Zed^{n+1}$. 
By Definition~\ref{definition:c},\, $c_F\leq \den(v)=d_F-c_F$. Let $m$
  be the unique integer 
such that $m \cdot  c_F<\den(v) \leq (m+1)  \cdot  c_F$. 
Also let  the rational point  $w\in\Rn$ be 
defined by 
    $\tilde w=\tilde v-m \cdot\tilde v_{e+1}$. 
       Observe that 
       $\{\tilde v_0,\ldots,\tilde v_{n-1},\tilde w\}$ is a base of $\Zed^{n+1}$.    
The  minimality property of  $c_F$\,entails
$c_F\leq\den(w)=\den(v)- m\cdot \den(v_{e+1})=\den(v)- m\cdot c_F\leq c_F$.
 Then  
$\den(v)=d_F-c_F=(m+1)  \cdot  c_F$, whence  $d_F=(m+2)\cdot c_F$. 
Since $\gcd(d_F,c_F)=1$, then $c_F=1$.

If   $e=n-1$, then $\den(v_n)=c_F$. 
 Let $v\in \Qn$ be such that $\tilde v=\tilde v_0-\tilde v_{n}$.
 Then 
$\{\tilde v_0,\ldots,\tilde v_{n-1},\tilde v\}$ is a base of $\Zed^{n+1}$. 
The  minimality property 
of  $c_F$ yields the inequality $c_F\leq \den(v)=d_F-c_F$, that is,  $c_F\leq d_F/2$.
\end{proof}

\begin{corollary}
Let $F,G\subseteq \Rn$ be non-empty rational affine spaces. Suppose $\dim(F)=\dim(G)\neq n-1$. Then there exists $\gamma\in\GLnZ$ such that $\gamma(F)=G$
iff $d_F=d_G$.
\end{corollary}

\begin{corollary}
Let $F,G\subseteq \Rn$ be non-empty rational affine spaces, with 
   $d_F=d_G\in\{1,2,3,4,6\}$. Then 
there exists $\gamma\in\GLnZ$ such that $\gamma(F)=G$
iff $\dim(F)=\dim(G)$.
\end{corollary}

The following result will find repeated use  in the sequel:

\begin{proposition}\label{Prop:Alternative}
Let $n\in\{1,2,\ldots\}$ and $e,d,c$ be  integers such that
\begin{itemize}
\item[(i)] $e\in\{0,\ldots,n-1\}$, $d\in\{1,2,\ldots\}$, and $c\in\{1,\ldots,\max(1,d/2)\}$;
\item[(ii)] $\gcd(d,c)=1$;
\item[(iii)] if $e\neq n-1$, then $c=1$.
\end{itemize}
Then for some integer
  $p\in \{1,\dots,d\}$ with $\gcd(p,d)=1$,   the rational affine space 
\[\textstyle F=\{(y_1,\ldots,y_n)\in\Rn\mid 
  y_{e+1}=\dots=  y_{n}
  =\frac{p}{d}\}
 \]
satisfies $(\dim(F),d_F,c_F)=(e,d,c)$.
\end{proposition}
\begin{proof}
By (ii), there exist $p\in \{1,\dots,d\}$ and $q=0,1,\dots$ such that
\[p\cdot c - q\cdot d=1.\]
The rational affine space  $F=\{(y_1,\ldots,y_n)\in\Rn\mid 
  y_{e+1}=\dots=  y_{n}
  =\frac{p}{d}\}$  satisfies $\dim(F)=e$ and  $\gcd(p,d)=1$.
Let the rational points $v_0,\ldots,v_{e}\in \Rn$ be defined 
 as follows, for each  $i\in\{0,\dots,e\}$ and $j\in\{1,\dots,n\}$:
\[
(v_i)_j= j \mbox{th coordinate of $v_i$}=\begin{cases}
\frac{p}{d}&\mbox{if } j\in\{e+1,\ldots, n\};\\
0&\mbox{if }j  \in\{1,\dots,e\}\setminus\{i\};\\
\frac{1}{d}&\mbox{if }  j = i.
\end{cases}
\]

\medskip
\noindent
Thus in particular, 
 $v_0=(\underbrace{0,\ldots,0}_{e \mbox{\tiny\,\, times}},
\underbrace{{p}/{d},\ldots,{p}/{d}}_{n-e \mbox{\tiny \,times}})$.
By definition,  $\den(v_0)=\cdots=\den(v_e)=d$. 

\medskip
To prove the identities $d_G=d$ and  $c_G=c$ we argue by cases.

\bigskip
\noindent{\it Case 1:} $c=1$.  We then set 
 $p=1$ and $q=0$. Let
 $\xi_1,\ldots,\xi_n$ denote the canonical base of the vector space $\Rn$,
  and  $\boldsymbol 0=(0,\ldots,0)$  the origin of $\Rn$.
  Then  the $n$-simplex
 $\conv(v_0,\ldots,v_e,\boldsymbol{0},\xi_{e+2},\ldots,\xi_{n})$ is  regular. 
 As a matter of fact,  the
   integer $(n+1)\times(n+1)$-matrix  $\mathcal T$ whose rows are
 the vectors $\tilde{v}_1, \ldots, \tilde{v}_e, \tilde{v}_0,
 \tilde{\xi}_{e+2},\ldots,\tilde{\xi}_{n}, 
 \tilde{\boldsymbol{0}},\,\,$ 
    is  lower triangular  with  all 
 diagonal elements  $=1$. By Lemma~\ref{Lemma:CalculationofD}, $d_F=\gcd(\den(v_0),\ldots,\den(v_e))=d$.  By Definition~\ref{definition:c},\, $c_F=1=c$.

\bigskip
\noindent{\it Case 2:} $c\neq 1$. By (iii), $e=n-1$. Let  the point $v\in\Rn $ be defined by
$
v=(0,\ldots,0, {q}/{c}).
$
Then  $\den(v) =c$.
Since  $p\cdot c-q\cdot d=1$,   the $n$-simplex
\, $\conv(v_0,\ldots,v_{n-1},v)$ is  
 regular. By Lemma~\ref{Lemma:CalculationofD}, $d_F=\gcd(\den(v_0),\ldots,\den(v_{n-1}))=d$.  By Definition~\ref{definition:c},\, $c_F\leq c$.
There is  $w\in\Qn$  such that $\den(w)=c_F$ and $\conv(v_0,\ldots,v_{n-1},w)$ is  
a regular $n$-simplex. 
Since  $\{\tilde{v}_1, \ldots, \tilde{v}_n,\tilde{w}\}$ 
is a basis of $\Zed^{n+1}$,
\,\,\, $\tilde{v}$ is a linear combination of 
$\{\tilde{v}_1, \ldots, \tilde{v}_n,\tilde{w}\}$ with
integer coefficients. 
Since $p\cdot c-q\cdot d=1$, 
we easily see that
 $c=c_F+k\cdot d$ for some $k\in\Zed$.  Using now
 (i) and  Lemma~\ref{lemma:expedient}
 we obtain $c_F=c$, which completes the proof.
\end{proof}

%%%%%%%%%%%%%%%%%%%%%%%%
\section{$(\GLnZ)$-orbit classification}\label{Section:Class}
%%%%%%%%%%%%%%%%%%%%%%%%

\noindent
For each $x=(x_1,\ldots,x_n)\in\Rn$ let
\[
\boxed{
F_x=\bigcap\{F\subseteq \Rn\mid x\in F\mbox{ and }F 
\mbox{ is a rational affine space}\}.
}
\]

\smallskip 
\noindent
Also let the group  $G_x$   be defined by 
\[
\boxed{
G_x=\Zed+x_1\Zed+\cdots+x_n\Zed.
}
\]

\smallskip 
\noindent
In other words,  $G_x$  is  the subgroup of the additive group
$\R$ generated by 1 together with the coordinates of $x$.

\begin{lemma}\label{Lemma:Necessary}
For each $x\in\Rn$ and  $\gamma\in\GLnZ$,  
$\gamma(F_x)=F_{\gamma(x)}$ and $ \,\,G_x=G_{\gamma(x)}.$
\end{lemma}

\begin{proof}
Let $y=\gamma(x)$.
Since $\gamma$ and $\gamma^{-1}$ preserve rational affine spaces we have
the inclusions
 $F_y\subseteq\gamma(F_x)$ and $F_x\subseteq \gamma^{-1}(F_y)$,
 whence  $F_y\subseteq\gamma(F_x)\subseteq\gamma(\gamma^{-1}(F_y))=F_y$.

Let $x=(x_1,\ldots,x_n)$ and $y=(y_1,\ldots,y_n)$.
 Let $\mathcal {U}
\in \Zed^{n\times n}$ and $b\in\Zed^n$  be such that 
$\gamma(z)=\mathcal {U}
z+b$
 for each $z\in \Rn$. Since 
all terms of $\mathcal {U}$ and all coordinates of  $b$ are integers, 
then $y_i\in G_x$ for each $i\in\{1,\ldots,n\}$. Therefore, 
$G_y\subseteq G_x$. 
The converse inclusion  follows replacing  $\gamma$ by~$\gamma^{-1}$.
\end{proof}

\begin{lemma} \label{Lemma:AffineSpaceClassFromPoint}
For each $x=(x_1,\ldots,x_n)\in\Rn$,
 \,
$\rank(G_x)=\dim(F_x)+1$ and\,\, $ d_{F_x}=\max\{k\in\Zed\mid 1/k\in G_x\}.$
\end{lemma}

\begin{proof}
Let $e=\dim(F_x)$  and  $d=d_{F_x}$. We argue by cases.
 
\medskip
\noindent {\it Case 1:  $e=n.$} 
Then   $F_x=\Rn$. Suppose the integers
 $k_0,\ldots,k_n$  satisfy
 $k_0+k_1\cdot x_1+\cdots+k_n\cdot x_n=0$. 
Letting the  vector $v\in \mathbb Z^n$ be defined by $v=(k_1,\ldots,k_n)$,
we have the identity
$\langle x,v\rangle=-k_0$.   Since
 $F_x=\Rn$ then necessarily  $k_1=\dots=k_n=0$, and hence
 $k_0=\langle x,v\rangle=0$. 
We have just proved 
 that $\rank(G_x)=n+1$, and  $G_x\cap \Q=\Zed$. 
It follows that $\max\{k\in\Zed\mid 1/k\in G_x\}=1$. 
By Definition~\ref{Def:DF},\, $d_{F_x}=1$, which settles the proof of 
this case.
\medskip

\noindent {\it Case 2:  $e<n.$}
Then   Proposition~\ref{Prop:Alternative} 
and  Theorem~\ref{Theorem:ClassAffineSpace}
jointly yield 
  $p\in \{1,\dots,d\}$ and  $\gamma\in\GLnZ$  
  such that $\gcd(p,d)=1$ and 
  $$\gamma(F_x)=\{(y_1,\ldots,y_n)\in\Rn\mid y_i
  =\frac{p}{d}\mbox{ for each }i=e+1,\dots,n\}=F.$$
 Let $\gamma(x)=y=(y_1,\ldots,y_e,\frac{p}{d},\ldots,\frac{p}{d})$.
 By Lemma~\ref{Lemma:Necessary}, 
 $G_y=G_x$ and $F_y=\gamma(F_x)=F$.
Let $k_0,\ldots,k_e$ be such that 
$k_0 \cdot \frac{1}{d}+k_1  \cdot y_1+\cdots+k_e \cdot  y_e=0$.
Then the rational point
$v=(k_1,\ldots,k_e,0,\ldots,0)$ satisfies 
$\langle v, y\rangle\in\Q$. 
Since $F_y= F$, we also have
 $\langle v, z\rangle\in\Q$  
for each $z\in F$,  
 whence  $v=0$ and
  $k_0 \cdot \frac{1}{d}+h_1 \cdot  y_1+\cdots+k_e  \cdot y_e=k_0 \cdot \frac{1}{d}=0$. 
Having thus proved that
 $k_i=0$ for each $i\in\{0,\ldots,e\}$,
 it follows that  $\rank(G_x)=\rank (G_y)=e+1$.

Assume the integer $k>0$ satisfies  $\frac{1}{k}\in G_x=G_y$. Then 
  there exist integers $k_0,\ldots,k_e$ 
satisfying  $\frac{1}{k}=k_0 \cdot \frac{1}{d}+k_1  \cdot y_1+\cdots+k_e \cdot y_e$. 
It follows that
 $1=k  \cdot k_0 \cdot  \frac{1}{d}+k   \cdot k_1 \cdot y_1+\cdots+k \cdot
 k_e \cdot y_e$.  From  $\rank(G_y)=e+1$ it follows that  
  $k \cdot k_0=d$. Thus $k\leq d$, and  $d=\max\{k\in\Zed\mid \frac{1}{k}\in G_x\}$.
  This settles Case 2 and completes
the proof. 
\end{proof}

We are now ready to prove the main result of this paper
and its corollaries, thus providing a complete classification
of $(\GLnZ)$-orbits.

\begin{theorem}
\label{theorem:classification}
For all $x,y\in\Rn$  the following conditions are equivalent:
\begin{itemize}
\item[(i)] there exists $\gamma\in\GLnZ$ such that $\gamma(x)=y$;
\item[(ii)] $(G_x,c_{F_x})=(G_y,c_{F_y})$.
\end{itemize}
\end{theorem}

\begin{proof}
(i)$\Rightarrow$(ii) follows directly from Theorem~\ref{Theorem:ClassAffineSpace} and Lemma~\ref{Lemma:Necessary}.
\medskip
(ii)$\Rightarrow$(i)    We argue by cases:

\medskip

\noindent{\it Case 1: $\rank(G_x)=n+1$.}
By Lemma~\ref{Lemma:AffineSpaceClassFromPoint},  $\dim(F_x)=\rank(G_x)-1=\rank(G_y)-1=\dim(F_y)=n$.
Since $G_x=G_y$, there are $(n+1)\times (n+1)$ integer matrices $\mathcal M$ and $\mathcal N$ such that $\mathcal M(x_1,\ldots,x_n,1)=(y_1,\ldots,y_n,1)$ and $\mathcal N(y_1,\ldots,y_n,1)=(x_1,\ldots,x_n,1)$. Thus $\mathcal M$ and $\mathcal N$ have the form
\[
\mathcal{M}=\left(\begin{tabular}{c|c} $\mathcal U$ & $b$ \\ \hline $0,\ldots,0 $ &
$1$ \end{tabular}\right) \quad \mathcal{N}=\left(\begin{tabular}{c|c} $\mathcal V$ & $c$ \\ \hline $0,\ldots,0 $ &
$1$ \end{tabular}\right) 
\]
and  
$\mathcal N\mathcal M(x_1,\ldots,x_n,1)=(x_1,\ldots,x_n,1)$. 
Since $\rank (G_x)=
n+1$,
\, $\mathcal N\mathcal M$ is the identity $(n+1)\times (n+1)$-matrix
  and  $\mathcal N=\mathcal M^{-1}$. Thus
  the affine map 
   $\gamma\colon \Rn\to \Rn$ defined by $\gamma(z)=\mathcal Uz+b$ 
   belongs to  $\GLnZ$ and $\gamma(x)=y$.

\bigskip

\noindent{\it Case 2: $1 \leq  \rank (G_x)  \leq n$.}
By Lemma~\ref{Lemma:AffineSpaceClassFromPoint},
 $\dim(F_x)=\rank(G_x)-1=\rank(G_y)-1=\dim(F_y)$, and 
$
d_{F_x}=\max\{k \in\Zed\mid 1/k \in G_x\}=
\max\{k \in\Zed\mid 1/ k\in G_y\}=d_{F_y}.
$
Let us use the
abbreviations
\[e=\dim(F_x)=\dim(F_y),\,\,\,
d=d_{F_x}=d_{F_y}, 
\,\, c=c_{F_x}=c_{F_y}.\]

\smallskip
\noindent
Proposition~\ref{Prop:Alternative} yields an integer $p$
with
 $1\leq p\leq d$ such that the affine space  
\[
\textstyle 
F=\{(y_1,\ldots,y_n)\in\Rn\mid y_i=\frac{p}{d}\mbox{ for each }
i\in \{e+1,\dots, n\}\}
\]
satisfies
 $(\dim(F),d_F,c_F)=(e,d,c)$. 
By Theorem~\ref{Theorem:ClassAffineSpace}
there are  $\gamma_1,\gamma_2\in\GLnZ$
  such that $\gamma_1(F_x)=\gamma_2(F_y)=F$.

\smallskip

Writing for short $x'=\gamma_1(x),\,\,y'=\gamma_2(y),$
the proof of (i) amounts to showing
\begin{equation}
\label{equation:tobeproved}
y'\in \orb(x').
\end{equation}

\noindent
To this purpose, we define
the rational points $v_0,\dots,v_e\in F$   by stipulating that for
all $i\in\{0,\dots,e\}$  and $j\in\{1,\dots, n\}$:
\[(v_i)_j=j\mbox{th coordinate of $v_i$}=\begin{cases}
\frac{p}{d}&\mbox{if } j\in\{e+1,\dots,n\};\\
0&\mbox{if }j  \in\{1,\dots,e\}\setminus\{i\};\\
\frac{1}{d}&\mbox{if } j = i.\
\end{cases}\]

\noindent
Then $F$ coincides with
$\aff(v_0,\ldots,v_e)$,  and $\conv(v_0,\ldots,v_e)$ is a regular
$e$-simplex.
Next let us define the
  map $\eta\colon F\to \R^{e}$  by 
$$
\eta(z_1,\ldots,z_n)=(d\cdot z_1,\ldots,d\cdot z_e)\mbox{ for all } z=(z_1,\ldots,z_n)\in F.
$$
Equivalently, letting $\xi_1,\ldots,\xi_e\in\R^{e}$ 
denote the vectors of the
 canonical base of  $\R^e$, \, $\eta$ is the unique affine 
transformation such that 
$\eta(v_0)=\boldsymbol{0}=$ the origin of $ \R^{e}$,
 and $\eta(v_i)=\xi_i$ for each 
$i\in\{1,\ldots, e\}$.
For any
integer
$m$ and  subgroup  $G$  of the additive group $\mathbb R$,
 let us use the self-explanatory notation
$
m\cdot G=\{m\cdot g\mid g\in G\}.
$
Then for every $z=(z_1,\ldots,z_n)\in F$ we can write
\begin{equation}\label{Eq:FlatGroup}
\textstyle
G_{\eta(z)}=d\cdot  z_1\Zed+\cdots +d\cdot  z_e\Zed+\Zed=d\cdot (z_1\Zed+\cdots +z_e\Zed+\frac{1}{d}\Zed)=d\cdot G_{z}.
\end{equation}
In particular,
$
G_{\eta(x')}=d\cdot G_{x'}=d\cdot G_x=d\cdot G_y=d\cdot G_{y'}= G_{\eta(y')}.
$
Since $\rank(G_{\eta(x')})=e+1=\rank(G_{\eta(y')})$,  
from Lemma
\ref{Lemma:AffineSpaceClassFromPoint} we obtain
 $\dim(F_{\eta(x')})=e$.
 The proof of  Case 1 above   (with $e$ in place of $n$) 
yields  $\gamma_3\in\GLeZ$ such that 
\begin{equation}\label{Eq:Gamma1}
\gamma_3(\eta(x'))=\eta(y').
\end{equation}
Since $\eta$ is one-to-one  onto $\R^e$, there exist uniquely
determined  rational points $w_0,\ldots,w_e\in F$ such that 
\begin{equation}\label{Eq:Ws}
\eta(w_0)=\gamma_3(\boldsymbol{0})\mbox{ and
} \eta(w_i)=\gamma_3(\xi_i)\mbox{ for each }i\in\{1,\ldots,e\}.\end{equation}
By  \eqref{Eq:FlatGroup},
\begin{equation}\label{Eq:DenWs}
 \den(w_i)=d  \cdot\den(\gamma_3(\xi_i))=d,
 \mbox{  for each $i\in\{1,\ldots,e\}$.}
\end{equation}

\bigskip
\noindent{\it Claim}: The $e$-simplex
$\conv(w_0,\ldots,w_e)\subseteq F$ is regular.
\smallskip

Indeed,  for each  $i\in \{1,\dots,e\}$ we have
$\widetilde{\eta(w}_i)=
\widetilde{\gamma_3(\xi}_i)=(\gamma_3(\xi_i),1)=(\eta(w_i),1)$.
Thus
\begin{equation}\label{Eq:BasisZe}
\{\widetilde{\eta(w}_0),\ldots,\widetilde{\eta(w}_{e})\}
 =\{(\eta(w_0),1),\ldots,(\eta(w_{e}),1)\}
 \mbox{ is a base of $\Zed^{e}$.}
\end{equation}

\smallskip
In case $c=1$,  the proof of 
Proposition~\ref{Prop:Alternative}
 shows that $p=1$,
whence 
\begin{equation}\label{EqWfromEtaW}
\tilde{w}_i=d\cdot (w_i,1)=(d\cdot w_i, d)
=(\eta(w_i),1,\ldots,1,d),
\mbox{ for each $i\in\{1,\ldots,e\}$.}
\end{equation}
 Let $\mu_1,\ldots,\mu_{n+1}$ denote the canonical base of the free
 abelian group 
$\Zed^{n+1}$.
From  \eqref{Eq:BasisZe}-\eqref{EqWfromEtaW}  it follows that
$\{\tilde{w}_0,\ldots,\tilde{w}_e,\mu_{e+2},\ldots,\mu_{n+1}\}$ 
is a base of $\Zed^{n+1}$. Thus
  the $e$-simplex
$\conv(w_0,\ldots,w_e)$ is regular.

\smallskip
In case   $c\neq 1$,   by   Lemma~\ref{lemma:expedient}, $e=n-1$.
Further, 
$\tilde{w}_i=d\cdot (w_i,1)=(d\cdot w_i, d)=(\eta(w_i),p,d)$ 
for each $i\in\{0,\ldots,e\}$. 
The proof of 
Proposition~\ref{Prop:Alternative}
provides an integer $q\geq 1$ such that 
$p\cdot c-q\cdot d=1.$ Upon 
defining  $v=(0,\ldots,0,
\frac{q}{c})$, from
  \eqref{Eq:BasisZe}
it follows that  the set of  integer vectors 
$
\{(\eta(w_0),p,d), \,\ldots, (\eta(w_e),p,d),\,\,(0,\ldots,0,q,c)\}
=\{\tilde{w}_0,\ldots,\tilde{w}_e,\tilde{v}\}
$
 is a base of $\Zed^{n+1}$.  
 Thus,  
$\conv(w_0,\ldots,w_e)$ is regular. 

Our claim is settled.

\medskip

As an immediate consequence of the
 regularity of the $e$-simplex $\conv(w_0,\ldots,w_e)$,  
 also the $n$-simplex 
$\conv(w_0,\ldots,w_e,v_{e+1},\ldots,v_n)$ is  regular. 
By \eqref{Eq:DenWs},\,\,\, $\den(w_i)=\den(v_i)=d$ for each 
$i\in\{0,\dots,e\}$.  Lemma~\ref{Lemma:GammaFromSimplex}
yields  $\gamma_{4}\in\GLnZ$ such that $\gamma_4(v_i)=w_i$ for 
$i\in\{0,\dots,e\}$, and
$\gamma_4(v_i)=v_i$  for 
$i\in\{e+1,\dots,n\}$.  
It follows that  $\gamma_4(F)=F$. 
 
Further, 
 \begin{equation}
 \label{Eq:CommutingGammasEta}
\eta(\gamma_4(z))=\gamma_3(\eta(z))
\mbox{ for each $z\in F$}.
\end{equation}  
As a matter of fact, 
 by \eqref{Eq:Ws}  we can write
$\eta(\gamma_4(v_0))=\eta(w_0)=\gamma_3(\boldsymbol{0})=\gamma_3(\eta(v_0))$
and  
$\eta(\gamma_4(v_i))=\eta(w_i)=\gamma_3(\xi_i)
=\gamma_3(\eta(v_i))$
for each $i=1,\dots,n$.  
Then  \eqref{Eq:CommutingGammasEta} follows, because 
 $\eta$, $\gamma_3$ and $\gamma_4$
are affine transformations  and \,\,$F=\aff(v_0,\ldots,v_e)$.
 
 Since   $\gamma_4(x')\in F$, 
from \eqref{Eq:Gamma1} and \eqref{Eq:CommutingGammasEta}
we can write 
$\eta(\gamma_4(x'))=\gamma_3(\eta(x'))=\eta(y').$
Since the map $\eta\colon F\to \R^e$ is one-to-one, we finally obtain  
$\gamma_4(x')=y',$
which proves \eqref{equation:tobeproved} and  completes
the proof of the theorem.
\end{proof}

Using Lemmas~\ref{lemma:expedient} 
and~\ref{Lemma:AffineSpaceClassFromPoint} we obtain the 
following special cases of Theorem~\ref{theorem:classification}:

\begin{corollary}
\label{corollary:semifinal}
For all $x,y\in\Rn$,  if $\rank(G_x)\neq n$, the following conditions are equivalent:
\begin{itemize}
\item[(i)] There exists $\gamma\in\GLnZ$, such that 
$\gamma(x)=y$;
\item[(ii)]$G_x=G_y$.
\end{itemize} 
\end{corollary}

\begin{corollary}
\label{corollary:final}
For all $x,y\in\Rn$, if  
$\max\{k\in\Zed\mid 1/k\in G_x\}\in\{1,2,3,4,6\}$, the following 
conditions are equivalent:
\begin{itemize}
\item[(i)] There exists $\gamma\in\GLnZ$, such that 
$\gamma(x)=y$;
\item[(ii)]$G_x=G_y$.
\end{itemize}
\end{corollary}

In conclusion, from  Lemmas~\ref{lemma:expedient} 
and~\ref{Lemma:AffineSpaceClassFromPoint}, 
Proposition~\ref{Prop:Alternative} and
Theorem \ref{theorem:classification} we obtain:
 
\begin{corollary}
\label{corollary:post-final} The map $ \orb(x)\mapsto (G_x,c_{F_x})$
is a one-one correspondence between $(\GLnZ)$-orbits
and pairs $(G,c)$, where $G$ is an arbitrary subgroup
 of the additive 
group $\mathbb R$ having rank $\leq n+1$ and containing
$1$ among its elements, 
 and $c$ is an integer satisfying the
following conditions: if $\rank(G)\not=n$ then $c=1;$
if $\rank(G)=n$,   letting $1/d$ be the smallest positive rational in $G$,
then $1\leq c\leq \max(1,d/2)$ and $\gcd(c,d)=1.$  There are
precisely  $\max(1,{\phi(d)}/{2})$ possible values for $c$.  
\end{corollary}

\end{document}